\documentclass[]{article}

\usepackage{amsfonts,amsmath,amssymb,amsthm,mathtools,mathrsfs}
\usepackage[utf8]{inputenc}
\usepackage[english]{babel}
\usepackage{bbm}
\usepackage{enumitem}
\usepackage{color}
\usepackage{comment}
\usepackage[margin=1.0in]{geometry}
\usepackage{setspace}
\newtheorem{theorem}{Theorem}

\newtheorem{proposition}[theorem]{Proposition}

\newtheorem{conjecture}[theorem]{Conjecture}
\newtheorem{corollary}[theorem]{Corollary}

\newtheorem{lemma}[theorem]{Lemma}

\theoremstyle{remark}
\newtheorem{remark}{Remark}

\numberwithin{theorem}{section}
\DeclareMathOperator{\Z}{\mathbb{Z}}

\newcommand{\bbN}{\mathbb{N}}

\usepackage[export]{adjustbox}                  % added
\protected\def\verythinspace{%
	\ifmmode
	\mskip0.5\thinmuskip
	\else
	\ifhmode
	\kern0.08334em
	\fi
	\fi
}

\title{Proof of the Diaconis--Freedman Conjecture on partially-exchangeable processes}
\author{Noah Halberstam and Tom Hutchcroft}

\begin{document}
	\maketitle

	\begin{abstract}
We prove a conjecture of Diaconis and Freedman (Ann.\ Probab.\ 1980) characterising the extreme points of the set of partially-exchangeable processes on a countable set. More concretely, we prove that the partially exchangeable sigma-algebra of any transient partially exchangeable process $X=(X_i)_{i\geq 0}$ (and hence any transient Markov chain) coincides up to null sets with the sigma-algebra generated by the initial state $X_0$ and the transition counts $( \#\{i\geq 0: X_i=x, X_{i+1}=y\} : x,y\in S)$. Our proof is based on an analysis of Gibbs measures for Eulerian paths on rooted digraphs, relying in particular on the connection to uniform spanning trees and Wilson's algorithm via the de Bruijn--Ehrenfest--Smith--Tutte (BEST) bijection, and yields an explicit method to sample from the conditional distribution of a transient Markov chain given its transition counts.
	\end{abstract}

\section{Introduction}

Given a countable set $S$, a random variable $X$ taking values in $S^\mathbb{N}$ (i.e., a random $S$-valued sequence) is said to be \textbf{exchangeable} if $(X_0,X_1,\ldots)$ has the same distribution as $(X_{\sigma(0)},X_{\sigma(1)},\ldots)$ for any permutation $\sigma:\bbN\to\bbN$ fixing all but finitely many elements of $\bbN$. Perhaps the best known theorem about exchangeable random sequences is \emph{de Finetti's theorem} \cite{de1937prevision} (see also \cite{diaconis1977finite,diaconis1980finite,gavalakis2021information,kirsch2019elementary}), which states that any exchangeable random sequence can be written as a \emph{mixture} of i.i.d.\ (independent and identically distributed) sequences. In other words, if $X$ is an exchangeable random sequence, we can sample $X$ by first sampling a random probability measure $\mu$ on $S$ and then sampling an i.i.d.\ sequence whose elements have law $\mu$. From the point of view of convex geometry, de Finetti's theorem is equivalent to the statement that the convex set of exchangeable probability measures on $S^\mathbb{N}$, denoted by $\mathscr{E}(S)$, has extreme point set $\operatorname{ext}\mathscr{E}(S)$ contained in the set of i.i.d.\ measures on $S^\mathbb{N}$. A refinement of de Finetti's theorem was later established by Hewitt and Savage \cite{hewitt1955symmetric}, who proved that these two sets are actually \emph{equal}: the extremal elements of $\mathscr{E}(S)$ are precisely the i.i.d.\ measures. 
This is equivalent to the probabilistic statement that i.i.d.\ random sequences satisfy a \emph{zero-one law} for exchangeable \emph{events}, a fact known as the \emph{Hewitt--Savage zero-one law}.
See e.g.\ \cite{aldous1985exchangeability,austin2008exchangeable,kallenberg2005probabilistic} for further background on exchangeable sequences and their applications and e.g.\ \cite{bayarri2004interplay,diaconis1977finite,sep-statistics,zabell1988symmetry} for discussions of the foundational role of de Finetti's theorem in Bayesian statistics.

% De Finetti \cite{deFinettiPartielle,de2011probabilitaa}
In their highly influential 1980 paper \cite{MR556418}, Diaconis and Freedman considered a natural generalization of exchangeability they called \emph{partial exchangeability}. 
Given two finite sequences $x=(x_0,x_1,\ldots,x_n)$ and $y=(y_0,y_1,\ldots,y_n)$ taking values in $S$, let us write $x \simeq y$ if $x_0=y_0$ and the transition counts $\#\{0\leq i \leq n-1 :(x_{i},x_{i+1})=(u,v)\}$ and $\#\{0\leq i \leq n-1 :(y_{i},y_{i+1})=(u,v)\}$ coincide for every $u,v\in S$. 
A random sequence $X=(X_0,X_1,\ldots)$ taking values in a countable set $S$ is said to be \textbf{partially exchangeable} if 
\[
\mathbb{P}(X_0,X_1,\ldots,X_n = x_0,x_1,\ldots, x_n)=\mathbb{P}(X_0,X_1,\ldots,X_n = x_{0}',x_{1}',\ldots, x_{n}')
\]
for every $n\geq 0$ and every pair of sequences $x=(x_0,\ldots, x_n)$ and $x'=(x_0',\ldots,x_n')$ in $S$ with $x\simeq x'$. 
In other words, a random sequence is partially exchangeable if the probability for some finite sequence to appear as an initial segment of the sequence depends only on the starting point and the number of transitions between each ordered pair of states.
  We denote the set of partially exchangeable measures on $S^\mathbb{N}$ by $\mathscr{P}(S)$. Every Markov chain on $S$ together with a choice of starting point defines a partially exchangeable random sequence on $S$, and indeed partial exchangeability is sometimes known as \emph{Markov exchangeability}. (Other notions of partial exchangeability have been studied in e.g.\ \cite{eagleson1978limit,aldous1981representations}.) The main result of \cite{hewitt1955symmetric} is that every \emph{recurrent} partially exchangeable sequence can be written as a mixture of recurrent Markov chains\footnote{An incorrect version of this theorem, without the hypothesis of recurrence, had been suggested in the work of de Finetti \cite{deFinettiPartielle,de2011probabilitaa}. See \cite{fortini2002mixtures} for further discussion.},
  % De Finetti also used a slightly different notion of partial exchangeability;  for further discussion.
   where a partially exchangeable process is said to be \textbf{recurrent} if it returns to its starting state infinitely often almost surely. This improved upon an earlier result of Freedman \cite{freedman1962mixtures}, which required the process to be stationary. Together with the Hewitt-Savage zero-one law, the Diaconis-Freedman theorem easily implies that the set of recurrent elements of $\operatorname{ext} \mathscr{P}(S)$ coincides with the set of (laws of trajectories of) recurrent Markov chains on $S$ with deterministic starting points. Thus, the Diaconis--Freedman theorem can be thought of yielding analogues of both de Finetti's theorem and the Hewitt--Savage theorem for \emph{recurrent} partially exchangeable processes, with i.i.d.\ sequences replaced by recurrent Markov chains.  The Diaconis-Freedman theorem was extended to \emph{continuous time} processes in the later work of Freedman \cite{freedman1996finetti}, quantitative versions of the theorem for finite processes were developed in the work of Zaman \cite{zaman1984urn,zaman1986finite}, and classifications of mixtures of various other kinds of processes in terms of their symmetries have been given in \cite{diaconis1988sufficiency,diaconis1984partial,diaconis1987dozen}. See \cite{DiaconisFreedmanPhilosophy} for a detailed discussion of the history of these results and of their consequences in the philosophy of Bayesian statistics.

The Diaconis--Freedman theorem has important applications in the study of processes with \emph{reinforcement}, such as edge-reinforced random walk and the vertex-reinforced jump process \cite{MR3189433,pemantle2007survey}, and has also inspired important work on the Bayesian analysis of Markov chains \cite{bacallado2011bayesian,diaconis2006bayesian,zabell1995characterizing}. As explained to us by Persi Diaconis, sampling algorithms originating in the study of finite partially-exchangeable processes have also been hugely influential in computational biology as they form a component of the BLAST algorithm, one of the most widely used algorithms for comparison and search of DNA and RNA sequences in bioinformatics \cite{karlin1990methods,pevzner2000computational}.

In \cite[Section 4]{MR556418}, Diaconis and Freedman made a conjecture concerning a counterpart of their theorem for \textbf{transient} partially exchangeable processes, that is, processes that visit each state at most finitely often. (In this case, it is \emph{not} true that every partially-exchangeable process can be written as a mixture of Markov chains, as follows by considering the deterministic partially-exchangeable sequence $0,0,1,1,2,2,3,3,\ldots$)
% , since conditioning a transient Markov chain $X$ on its transition numbers $(\#\{i \geq 0 : (X_i,X_{i+1})=(u,v)\})_{u,v\in S}$ yields a partially-exchangeable process with non-random transition numbers which is typically not expressible as a mixture of Markov chains.
 We will state their conjecture probabilistically for now (as in their original paper) and explain its interpretations in terms of convex geometry and mixture representations in Section~\ref{subsec:mixtures}.

\begin{conjecture}[Diaconis--Freedman 1980]
Let $X$ be a transient partially exchangeable process taking values in a countable set $S$ with non-random initial state $X_0$. Then the partially exchangeable $\sigma$-algebra of $X$ is generated up to null sets by the transition numbers $M=(M(u,v))_{u,v\in S}=(\#\{i \geq 0 : (X_i,X_{i+1})=(u,v)\})_{u,v\in S}$.
\end{conjecture}

Here, we define the equivalence relation $\simeq$ for infinite sequences by setting $x\simeq x'$ if and only if there exists $N<\infty$ such that $x_n=x'_n$ for all $n\geq N$ and $(x_0,\ldots,x_N)\simeq (x_0',\ldots,x_N')$, say that
 a set $A \subseteq S^\bbN$ is \textbf{partially exchangeable} if 
\[x \in A \Rightarrow x' \in A \qquad \forall x' \simeq x,\] 
% whenever $x=(x_0,x_1,\ldots)$ is a sequence in $S$,
 and define the 
 \textbf{partially exchangeable sigma-algebra} $\mathcal{P}(S)$ to be the sigma-algebra of partially exchangeable Borel subsets of $S^\bbN$.
 Equivalently, the partially exchangeable sigma-algebra is equal to $\bigcap_{n\geq 1} \mathcal{P}_n(S)$, where $\mathcal{P}_n(S)$ is the sigma-algebra in $S^\bbN=\{(x_0,x_1,\ldots):x_0,x_1,\ldots \in S\}$ generated by the initial state $x_0$, the tail $(x_m)_{m \geq n}$, and the transition numbers $(m_n(u,v))_{u,v\in S}$ defined by $m_n(u,v)=\#\{0\leq i \leq n-1 :(x_{i},x_{i+1})=(u,v)\}$. It is easily seen
 % (see Lemma [ref])
  that a measure in $\mu\in \mathscr{P}(S)$ is extremal if and only if $\mathcal{P}(S)$ is trivial with respect to $\mu$, meaning that $\mu(A)\in\{0,1\}$ for every $A\in \mathcal{P}(S)$. Thus, the Diaconis--Freedman conjecture is equivalent to the statement that a measure in $\mathscr{P}(S)$ is extremal if and only if it is supported on configurations with a constant value of the initial state and the crossing numbers (see the proof of Corollary~\ref{cor:extreme_points} for more details). Since the partially exchangeable sigma-algebra $\mathcal{P}(S)$ always includes the tail sigma-algebra $\mathcal{T}(S)=\bigcap_{n\geq 1} \langle x_n,x_{n+1},\ldots \rangle$, the Diaconis--Freedman conjecture also implies that every transient partially-exchangeable process (and hence every transient Markov chain) has tail sigma-algebra contained in the completion of the sigma-algebra generated by its transition numbers.
% In the next section we give a more explicit description of these measures as Gibbs measures on Eulerian paths.
% \[
% \bigcap_{n\geq 1} \langle x_0, \#\{0\leq i \leq n :(x_{i},x_{i+1})=(u,v)\}, (x_m)_{m > n} \rangle.
% \]

% \medskip

The case of the Diaconis--Freedman conjecture in which $X$ is a transient \emph{Markov chain} has received especially significant attention over subsequent years. This interest was motivated in part by its close analogy with a conjecture on the Poisson boundaries of Lamplighter groups and free metabelian groups posed by Kaimanovich and Vershik \cite{MR704539,vershik1979random} and solved by Erschler \cite{erschler2010poisson,erschler2010poisson} and Lyons and Peres \cite{lyons2020poisson}, and to a problem of Kaimanovich \cite{kaimanovich1991poisson} asking for which Markov chains the \emph{exchangeable} sigma-algebra is generated by the \emph{vertex} local times. In this context, most progress has taken the form of sufficient conditions for a transient Markov chain to have infinitely many \emph{cut times} almost surely \cite{MR2789585,MR2644879,halberstam2023most,MR1687097,MR1423466,dur34690}, which in turn is a sufficient condition for the partially exchangeable sigma-algebra to be generated up to null sets by the transition numbers \cite{MR1687097}. (Here, a \emph{cut time} for the process $X$ is a time $t$ such that $\{X_i :i < t\}$ and $\{X_i : i \geq t\}$ are disjoint.) In particular, it is now known that every transient random walk on a group generated by a finitely-supported measure has infinitely many cut-times a.s.\ \cite{MR1687097,MR1983173}, and the authors \cite{halberstam2023most} recently proved that Markov chains have infinitely many cut times a.s.\ under assumptions only barely stronger than transience (e.g.\ spectral dimension strictly larger than $2$). On the other hand, there do exist transient Markov chains with only finitely many cut-times almost surely (indeed, there exist examples of both 
% so-called 
birth-death chains \cite{james2008transient} and bounded-degree graphs \cite{MR2789585} with this property),
 % -- Markov chains on the natural numbers making only nearest-neighbour jumps),
  and as far as we know the Diaconis--Freedman conjecture has remained open even for, say, sums of i.i.d.\ integer random variables  without any assumptions on their tail decay.
% in the barley transient regime (e.g., sums of i.i.d.\ random variables with $t^{-1}(\log t)^\alpha$ tails).
The only work we are aware of that resolves a case of the conjecture \emph{without} using cut times is that of James \cite{JamesThesis} (published in \cite{james2008transient}), who proved the conjecture for nearest-neighbour random walks on trees; for such random walks the Diaconis--Freedman conjecture and Kaimanovich's question are equivalent since the vertex local times determine the transition counts.

\medskip

In this paper we completely resolve the Diaconis--Freedman conjecture.

\begin{theorem}
\label{thm:DF_main}
Let $X$ be a transient partially exchangeable process taking values in a countable set $S$ with non-random initial state $X_0$. Then the partially exchangeable $\sigma$-algebra of $X$ is generated up to null sets by the transition numbers of $X$.
\end{theorem}

As explained above, this theorem is equivalent to the statement that for any assignment of initial state and transition numbers, there is at most one partially exchangeable process taking those transitions almost surely. Our proof of this theorem not only establishes the abstract uniqueness of such processes, but 
yields an explicit description of them in terms of a directed analogue of the \emph{wired uniform spanning forest} \cite{Pemantle,benjamini2001special}, which can be sampled using \emph{Wilson's algorithm} \cite{wilson1996generating}. See Theorem~\ref{thm:DF_explicit} for this explicit description. Our proof that the processes we construct using Wilson's algorithm are the \emph{only} extremal transient partially exchangeable processes builds on ideas first introduced to prove strong forms of ergodicity and tail triviality for the (undirected) wired uniform spanning forest known as \emph{indistinguishability theorems} \cite{hutchcroft2020indistinguishability,hutchcroft2017indistinguishability}. The connection between partially exchangeable processes and uniform spanning trees derives from the \emph{de Bruijn--Ehrenfest--Smith--Tutte \emph{(BEST)} bijection} for Eulerian paths \cite{van1987circuits,tutte1941unicursal}, which we review in Section~\ref{subsec:Euler_background}.

\subsection{Uniqueness of proper Gibbs measures on Eulerian paths}

To prove Theorem~\ref{thm:DF_main}, we will reformulate the conjecture as a statement about \emph{uniqueness of Gibbs measures} on Eulerian paths in sourced Eulerian digraphs. We begin by giving some relevant definitions. 

\medskip

A \textbf{digraph} (a.k.a.\ directed graph) $G=(V,E)$ consists of a countable set of vertices $V$ and a countable set of edges $E$ together with a function $E\to V\times V$ denoted by $e \mapsto (e^-,e^+)$ mapping each edge to its tail $e^-$ and its head $e^+$. (In particular, we allow for edges with head equal to tail and for multiple edges with the same head and tail.) For each vertex $v\in V$ we write $E^\rightarrow_v=\{e\in E: e^-=v\}$ for the set of edges pointing out of $v$ and write $E^\leftarrow_v=\{e\in E:e^+=v\}$ for the set of edges pointing into $v$; the cardinalities $|E^\rightarrow_v|$ and $|E^\leftarrow_v|$ are called the \textbf{out-degree} and \textbf{in-degree} of $v$, respectively. 
We say that $G$ is \textbf{locally finite} if every vertex has finite in-degree and out-degree. A \textbf{path} in a digraph is a sequence of vertices $(\gamma_i)_{i=n}^m$ and edges $(\gamma_{i,i+1})_{i=n}^{m-1}$ with $n,m\in \Z \cup \{-\infty,\infty\}$ such that $(\gamma_{i,i+1}^-,\gamma_{i,i+1}^+)=(\gamma_i,\gamma_{i+1})$ for every $n \leq i < m$; if $n=0$ and $m$ is finite we say that the path has starting point $e_0^-$ and final point $e_m^+$. We say that a path is \textbf{transient} if it visits each vertex at most finitely many times (in particular, finite paths are always transient).
We say that $G$ is \textbf{connected} if it is connected as an undirected graph, and that $G$ is \textbf{irreducible} if any vertex can be reached from any other vertex by a directed path.

\medskip

\noindent \textbf{Sourced Eulerian digraphs.} We will be particularly interested in the following classes of digraphs:
\begin{itemize} 
\item We define a \textbf{finite digraph with source and sink} to be a triple $(G,o,z)$ consisting of locally finite digraph $G=(V,E)$ together with a distinguished source vertex $o$ and a distinguished sink vertex $z$ such that for each vertex $u\in V$ there exists a path from $o$ to $u$ and from $u$ to $z$. We say that a finite digraph with source and sink $(G,o,z)$ with $o\neq z$ is \textbf{Eulerian} if either every vertex in $V\setminus \{o,z\}$ has in-degree equal to out-degree, the root $o$ has out-degree one higher than in-degree, and $z$ has in-degree one higher than out-degree; if $o=z$ we instead require that \emph{every} vertex has in-degree equal to out-degree. 
It is a classical theorem (the \emph{undirected} analogue of which was famously proven by Euler in 1735 \cite{euler1741solutio}) that a finite digraph $G$ is Eulerian with source $o$ and sink $z$ if and only if there is an \textbf{Eulerian path} from $o$ to $z$ in $G$, that is, a path that starts at $o$, ends at $z$, and crosses every edge of $G$ exactly once \cite[Chapter 13]{polya2009notes}.
\item We define a \textbf{infinite sourced digraph with sink at infinity}, which we will usually refer to simply as an infinite sourced digraph, to be a pair $(G,o)$ consisting of an infinite, connected, locally finite digraph $G$ and a distinguished source vertex $o$.
% with the property that for every vertex $u$ in $V$ there is a path starting at $o$ and ending at $u$.
 We say that an infinite sourced digraph $(G,o)$ is \textbf{Eulerian} if every vertex other than $o$ has in-degree equal to out-degree, while the source vertex $o$ has out-degree one higher than in-degree. 
 \end{itemize}

Given an infinite sourced digraph $(G,o)$ with sink at infinity, we define a \textbf{proper Gibbs measure on Eulerian paths from $o$ to $\infty$} in $G$ to be a probability measure $\mu$ on the space of paths in $G$ such that the following hold:
\begin{enumerate}
	\item $\mu$ is supported on half-infinite paths in $G$ indexed by $\{0,1,2,\ldots\}$ that start at $o$ and cross every edge of $G$ exactly once.
	\item If $X$ is a random path with law $\mu$, then the conditional distribution of the first $n$ steps of $X$ given the steps the walk takes after time $n$ is a uniform random Eulerian path from $o$ to $X_n$ in the graph with edge set $\{X_{i,i+1} : 0\leq i < n\}$. (Note that this edge set is equal to the set of edges \emph{not} crossed \emph{after} time $n$, and hence is determined by the steps taken by the walk after time $n$.)
\end{enumerate}
The second condition in this definition implies in particular that the sequence of vertices visited by a random path distributed as a proper Gibbs measure on Eulerian paths is partially exchangeable.

\begin{remark}
\label{remark:proper}
The word ``proper'' is used to distinguish this class of Gibbs measures from larger classes which are not necessarily supported on single paths. For example, one could define a Gibbs measure on Eulerian paths in a sourced Eulerian digraph to be any measure obtained as a weak limit (in an appropriate topology) of uniform Eulerian path measures on finite approximations to $G$. A similar definition for \emph{arborescences} is given in detail in Section~\ref{subsec:Gibbs_uniqueness}. This class of Gibbs measures can also be characterised axiomatically using an approach analogous to that of \cite{halberstam2023uniqueness}; we do not pursue these matters here.  In the limit, such measures could be supported on \emph{collections} of paths which \emph{together} cross every edge exactly once, analogously to how Gibbs measures for the uniform spanning tree may be supported on forests rather than trees \cite{Pemantle}. (Such collections can be encoded locally using \emph{stack configurations} along the lines discussed in Section~\ref{subsec:Euler_background}.) 
\end{remark}

We will prove the following reformulation of Theorem~\ref{thm:DF_main}.

\begin{theorem}
\label{thm:DF_Gibbs}
Let $(G,o)$ be a sourced Eulerian digraph. There is at most one proper Gibbs measure on Eulerian paths from $o$ to $\infty$ in $G$. 
\end{theorem}

\begin{proof}[Proof of Theorem~\ref{thm:DF_main} given Theorem~\ref{thm:DF_Gibbs}]
% It suffices to prove that every transient partially-exchangeable process with non-random initial state and non-random transition numbers has trivial partially-exchangeable $\sigma$-algebra; the general case of the theorem follows by conditioning on the transition numbers and observing that this preserves partial exchangeability. To prove this claim, 
% 
Observe that if  $X$ is a partially exchangeable process with non-random transition numbers $M(u,v)$, we can turn it into a proper Gibbs measure on Eulerian paths on the sourced Eulerian digraph with vertex set $\{u\in S: \text{$u$ visited a.s.\ by $X$}\}$ and with $M(u,v)$ directed edges from $u$ to $v$ for each $u,v\in S$: This is done by choosing a random order for the process to cross the $M(u,v)$ different edges from $u$ to $v$ over the course of its transitions from $u$ to $v$. Thus, Theorem~\ref{thm:DF_Gibbs} is equivalent to the statement that for each assignment of an initial vertex and transition numbers between the points of some countable set $S$, there is at most one partially-exchangeable process that takes these transition numbers almost surely. This implies that every transient partially-exchangeable process with non-random initial state and non-random transition numbers has trivial partially-exchangeable $\sigma$-algebra, since otherwise we could condition on a non-trivial partially-exchangeable event to get a different partially-exchangeable process with the same transition numbers.
The general case of the theorem follows by conditioning on the transition numbers, which preserves partial exchangeability.
% Suppose the theorem is false, so that there exists a transient partially-exchangeable process $X$ taking values in a countable set $S$ with non-random initial state $X_0$ such that the partially exchangeable $\sigma$-algebra $\mathcal{P}$ of $X$ does not coincide up to null sets with the $\sigma$-algebra $\mathcal{T}$ generated by the transition numbers of $X$. Let $A$ be an event in $\mathcal{P}$ that is not in the completion of $\mathcal{T}$. 
\end{proof}

As mentioned above, the proof of Theorem~\ref{thm:DF_Gibbs} establishes not only that proper Gibbs measures are unique, but yields an explicit description  of these measure in terms of a directed analogue of the wired uniform spanning forest (see Theorem~\ref{thm:DF_explicit}). This construction makes sense on \emph{any} sourced Eulerian digraph, but does not always yield an Eulerian path: As in Remark~\ref{remark:proper}, it may instead yield one half-infinite path together with a family of bi-infinite paths which \emph{together} cross each edge exactly once. It remains an interesting open problem to characterize those sourced Eulerian digraphs that admit at least one proper Gibbs measure on Eulerian paths (which is necessarily unique by Theorem~\ref{thm:DF_Gibbs}).

% \begin{problem}[The existence problem]
% Characterise the family of sourced Eulerian digraphs that admit at least one proper Gibbs measure on Eulerian paths.
% \end{problem}

% \begin{theorem}
% Each infinite sourced Eulerian digraph $(G,o)$ admits at most one proper Gibbs measure on Eulerian paths. 
% % Moreover, if such a measure exists it is both time-tail trivial and space-tail trivial.
% \end{theorem}

% \begin{theorem}
% Every countably infinite, locally finite digraph with sink at infinity admits at most one Gibbs measure on spanning in-arborescences rooted at infinity that is supported on arboreta in which every arborescence is one-ended. 
% \end{theorem}

\subsection{Extreme points and mixing representations}
\label{subsec:mixtures}

We now explain how our main results may be interpreted in terms of convex geometry and mixture representations. We begin by characterising the extreme points of the set of transient partially exchangeable processes:

\begin{corollary}
\label{cor:extreme_points}
Let $S$ be a non-empty countable set, and let $\mu$ be a transient, partially exchangeable measure $\mu$ on $S^\mathbb{N}$. The following are equivalent:
\begin{enumerate}
\item $\mu$ is extremal in $\mathscr{P}(S)$.
\item $\mu$ is supported on a set of sequences with constant initial vertex and constant transition numbers.
\item $\mu$ is equal to the measure on sequences associated to the unique proper Gibbs measure on Eulerian paths on some sourced Eulerian digraph defined on a subset of $S$.
\end{enumerate}
\end{corollary}

\begin{proof}[Proof of Corollary~\ref{cor:extreme_points}]
The implication $(1)\Rightarrow (2)$ is trivial since conditioning a transient partially exchangeable process on its initial vertex and transition numbers preserves partial exchangeability. The implication $(2) \Rightarrow (3)$ is the content of Theorem~\ref{thm:DF_Gibbs}. Finally, the implication $(3)\Rightarrow (1)$ follows from the uniqueness statement that is part of $(3)$: if the measure were not extremal then we could write it as the convex combination of two distinct partially exchangeable processes which would have the same deterministic initial vertex and transition numbers, contradicting the claimed uniqueness.
\end{proof}

Finally, we have the following theorem on extreme points and mixture representations for arbitrary partially exchangeable processes. 

\begin{corollary}
\label{cor:mixture_representations}
Let $S$ be a non-empty countable set. The set $\operatorname{ext}\mathscr{P}(S)$ of extremal elements of the set of partially exchangeable random sequences in $S$ can be written as the union of the following three sets:
\begin{enumerate}
	\item (Recurrent) The set of laws of trajectories of recurrent Markov chains on subset of $S$ with non-random initial state.
	\item (Transient) The set of proper Gibbs measures on Eulerian paths on  sourced Eulerian digraphs defined on subsets of $S$.
	\item (Mixed) The set of laws of processes formed by first following a uniform random Eulerian path on a finite Eulerian digraph defined on a subset of $S$ with a source and a sink, followed by the trajectory of a recurrent Markov chain defined on a disjoint subset of $S$ with non-random initial state.
\end{enumerate}
Moreover, every measure in $\mathscr{P}(S)$ can be written \emph{uniquely} as a mixture of measures in $\operatorname{ext}\mathscr{P}(S)$.
\end{corollary}

% (Note that an erroneous statement about extremal measures in the mixed case is made in \cite{MR556418}, where it is written that the initial sequence of transient states must be deterministic rather than a uniform Eulerian path on a finite digraph.)

\begin{proof}[Proof of Corollary~\ref{cor:mixture_representations}]
We first prove the claim about extreme points. The recurrent case is handled by the original paper of Diaconis and Freedman \cite{MR556418}, while the transient case is handled by Corollary~\ref{cor:extreme_points}. We now treat the mixed case, in which some states are transient and others are recurrent.  We first note that, with probability one, there is no transient state that is visited after a recurrent state. Indeed, if $\gamma = (s_0,\ldots,s_k)$ is a finite sequence such that $s_0$ and $s_k$ are recurrent and $s_1,\ldots,s_{k-1}$ are transient then the probability that the process follows the path $\gamma$ immediately after its $n$th visit to $s_0$ is independent of $n$ by partial exchangeability, but must converge to zero as $n\to\infty$ since $s_0$ is recurrent and $s_1$ is transient. This implies that no such path is ever taken a.s., so that no transient state is visited after a recurrent state a.s.\ as claimed. Thus, an extremal partially exchangeable process with both recurrent and transient states must first visit some finite deterministic collection of transient states, with deterministic crossing numbers, then spend all subsequent steps at recurrent states, with a deterministic choice of the first recurrent state visited. The law of the initial transient part of the process must be a uniform random Eulerian path on a finite digraph with source and sink by partial exchangeability, while the law of the recurrent part of the process is itself an extremal partially exchangeable process, hence the law of a recurrent Markov chain with non-random initial state by the Diaconis--Freedman theorem.

We now prove the claim about mixture representations. The fact that every measure in $\mathscr{P}(S)$ may be written as a mixture of measures in $\operatorname{ext}\mathscr{P}(S)$ follows by Choquet's theorem \cite{phelps2001lectures} since $\mathscr{P}(S)$ is convex and weakly compact if we give $S$ the topology making it the one point compactification of $S\setminus \{s_0\}$ for some $s_0\in S$. (Alternatively, the mixture representation is also given by conditioning on the partially-exchangeable sigma-algebra.) The fact that the mixture representation is \emph{unique} follows since it can be computed from the joint law of the transition numbers $(M(u,v)=\#\{$transitions from $u$ to $v))_{u,v\in S}$
and the transition frequencies
\[
Q(u,v) = \lim_{n\to\infty} \frac{\#\{0\leq i \leq n: X_i = u,X_{i+1}=v\}}{\#\{0\leq i \leq n: X_i = u\}},
\]
which are well-defined a.s.\ for every pair of states that are visited infinitely often by the process (as follows by applying the ergodic theorem to the recurrent Markov chains that appear in the mixture representation).
\end{proof}

\begin{remark}
In the terminology of convex geometry, the fact that every point in $\mathscr{P}(S)$ admits a \emph{unique} integral representation in terms of the extreme points of $\mathscr{P}(S)$ means that $\mathscr{P}(S)$ is a \emph{Choquet simplex}.
\end{remark}

\section{Background}

\subsection{Stacks, arborescences, and the BEST theorem}
\label{subsec:Euler_background}

\noindent \textbf{Encoding paths with stacks.} 
Given a digraph $G$, a \textbf{stack configuration} $S$ on $G$ is an assignment of each vertex $v$ of $G$ to a (finite or infinite) list of edges $S(v)=(S_i(v))_{i=1}^{n_v}$ in $E^\rightarrow_v$, which we call the \emph{stack at $v$}, where $n_v\in \{0,1,2,\ldots\}\cup \{\infty\}$; if $n_v=0$ then the stack at $v$ is the empty list. We denote the space of all stack configurations on $G$ by $\mathcal{S}=\mathcal{S}(G)$, and equip this space with the prodict topology. We say that a stack configuration $S$ is \textbf{at the top} of a stack configuration $S'$ if the list $S(v)$ is an initial segment of the list $S'(v)$ for each $v\in V$. If the stack configuration $S$ is at the top of the stack configuration $S'$, we define $S'-S$ by deleting the initial segment $S(v)$ from each of the lists $S'(v)$. Given an edge $e\in E$, we define a stack configuration $S^e$ by taking $S^e(e^-)$ to be the list of length $1$ with single entry equal to $e$ and taking $S^e(v)$ to be the empty list for every $v\neq e^-$, and say that $e$ is at the top of a stack configuration $S$ if $S^e$ is at the top of $S$, i.e., if $e$ is the first element in the stack at its tail; when this holds, we define $S-e:=S-S^e$.
Given a transient path $P$ in $G$, we define a stack configuration $S^P$ in $G$ by, for each vertex $v\in V$, letting $S^P(v)$ be the list of edges in $E^\rightarrow_v$ in the order that they are crossed by $P$, with edges that are crossed more than once appearing multiple times in the list as appropriate. We say that a path $P$ is at the top of a stack configuration $S$ if $S^P$ is, in which case we define $S-P:=S-S^P$.

For each stack configuration $S$ and initial vertex $v_0\in V$, there is a unique maximal path $P(v_0,S)$ starting at $v_0$ that is at the top of $S$, in the sense that any other path that starts at $v_0$ and it at the top of $S$ is an intial segment of $P(v_0,S)$. The path $P(v_0,S)$ is constructed from $S$ by ``following the stacks and throwing away each edge from the top of the stack after it is used'', a procedure we now define formally.
 The construction will simultaneously construct a sequence of vertices $v_0,v_1,\ldots$ and stack configurations $S=S^0,S^1,\ldots$, with $v_i$ equal to the $i$th vertex visited by the path. Initialize the algorithm by setting $S^0=S$ and setting $P^0$ to be the empty path.  
Suppose we have run the algorithm for $i$ steps, yielding a vertex $v_i$, a stack configuration $S^i$, and a path $P^i=(P_1,\ldots,P_i)$. If the stack above $v_i$ in $S^i$ is empty, we terminate the algorithm; we have already found the maximal path $P^i$. Otherwise, we set $P_{i+1}$ to be the edge on the top of the stack $S^i(v_i)$, define $P^{i+1}$ by appending $P_{i+1}$ to the end of $P^i$, set $v_{i+1}=P_{i+1}^+$, and set $S^{i+1}=S^i - P_{i+1}$. Note that if $P$ is a path starting at a vertex $v_0$ then $P=P(v_0,S^P)$, so that every path can be recovered from its associated stack and initial vertex. On the other hand, not every stack configuration $S$ encodes a single path: It is possible (and, indeed, typical) that no matter which vertex we start at, the path $P(v_0,S)$ terminates before every stack has been emptied.

\begin{remark}
Markov chains can be sampled using stack configurations by taking the stack above each vertex $v$ to be an infinite list of i.i.d.\ random elements of $E^\rightarrow_v$; the law of each such random element determines the transition probabilities of the Markov chain. This perspective on Markov chains is often used to prove the validity of Wilson's algorithm for sampling uniform spanning trees and arborescences, which we will discuss later in the paper. The paths obtained from other interesting laws on stack configurations have also been studied under the name of the \emph{rotor-router model} \cite{holroyd2008chip,chan2021infinite}.
\end{remark}

% \medskip

\noindent \textbf{Arborescences.} 
Given a digraph $A$ and a vertex $v$ of $A$, we say that $A$ is an \textbf{in-arborescence rooted at $v$} if every vertex in $V\setminus \{v\}$ has out-degree $1$, $v$ has out-degree zero, and every vertex in $V \setminus \{v\}$ is the initial vertex of a path in $A$ ending at $v$ (which is necessarily formed by following the unique outgoing edges from each vertex until $v$ is reached).
We say that an infinite digraph $A$ is an in-arborescence 
% (resp.\ out-arborescence)
 \textbf{rooted at infinity} if every vertex has out-degree $1$ 
 % (resp.\ in-degree $1$)
  and the tree formed by forgetting the orientations of the edges of $A$ is connected. If $A$ satisfies all properties required of an in-arborescence rooted at infinity 
  % (resp.\ out-arborescence rooted at infinity)
   other than connectivity, we call it an \textbf{in-arboretum} 
   % (resp. \textbf{out-arboretum})
    rooted at infinity. Given a digraph $G=(V,E)$, we define a \textbf{spanning in-arborescence of $G$}, 
    either rooted at $v\in V$ or at infinity, as a subdigraph of $G$ which is an in-arborescence, rooted at either $v$ or infinity as appropriate, and contains all the vertices of $G$. Spanning in-arboreta of $G$ rooted at infinity are defined similarly.

\medskip

\noindent \textbf{Arborescences from paths.}
Given a digraph $G=(V,E)$ and a transient path $P$ in $G$, we define the \textbf{last-exit-arboretum} $\mathscr{L}(P)$ of $P$
to be the digraph with vertex set equal to the set of vertices visited by $P$ and where an edge $e\in E$ is included in the edge set of $\mathscr{L}(P)$ if and only if it is the edge traversed by $P$ as it exits $e^-$ for the last time. 
(If $P$ is finite, no edge pointing out of the last vertex visited by $P$ is included.) We will often refer to the last-exit arboretum as the last-exit \emph{arborescence} in situations when we know it to be connected.
% 
% Similarly, if $P$ is backward-transient, we define the \textbf{first-entrance-arboretum} $\mathscr{F}(P)$ to be the digraph with vertex set equal to $P$ and where an edge $e\in E$ is included in the edge set of $\mathscr{F}(P)$ if and only if it is the edge traversed by $P$ as it enters $e^+$ for the first time. (If $P$ is backward-finite, no edge pointing into the initial vertex of $P$ is included.)
% 
% 

\medskip

\noindent \textbf{Ends and the past.} 
Given an in-arboretum $A=(V,E)$, we define the \textbf{past} $\mathrm{Past}_{A}(u)$ of a vertex $u\in V$ to be the set of vertices $x\in V$ such that there exists a directed path in $A$ starting at $x$ and ending at $u$, and the \textbf{future} $\mathrm{Future}_{A}(u)$ of a vertex $u$ to be the set of vertices $x\in V$ such that $u\in \mathrm{Past}_{A}(x)$.
We say that an infinite in-arborescence rooted at infinity $A$ is \textbf{one-ended} if the past $\mathrm{Past}_{A}(u)$ of each vertex $u\in V$ is finite. 

\medskip

The last-exit arboretum has the following important property.

\begin{lemma}
\label{lem:entrance_exit_basic_properties}
	Let $G=(V,E)$ be a digraph and let $P$ be a path in $G$. If $P$ is infinite, transient, and indexed by $\{0,1,2,\ldots\}$ then the last-exit-arboretum $\mathscr{L}(P)$ is an in-arboretum rooted at infinity, all of whose components are one-ended.
\end{lemma}

\begin{proof}
In order for a vertex $u$ to be in the past of a vertex $v$, it must be visited by $P$ before $v$. The claim follows  since $P$ visits at most finitely many vertices before it visits any specific vertex along its path.
\end{proof}

\noindent \textbf{The BEST theorem.} As we have seen in Lemma~\ref{lem:entrance_exit_basic_properties}, the last-exit edges of a finite path always define an in-arborescence rooted at the final vertex of the path. Thus, in order for a finite stack configuration to encode a finite path, the collection of edges at the end of each list in the stack must form an in-arborescence.
The BEST theorem, due to de Bruijn, Ehrenfest, Smith, and Tutte \cite{tutte1941unicursal,van1987circuits}, builds on this observation to give a simple characterisation of stack configurations encoding Eulerian paths. See e.g.\ \cite{stanley1999enumerative} for a modern treatment.

\begin{theorem}[BEST]
Let $(G,o,z)$ be a finite Eulerian digraph with source $o$ and sink $z$. A stack configuration $S$ on $G$ represents an Eulerian path from $o$ to $z$ if and only if the following hold:
\begin{enumerate}
	\item For each vertex $v$, the list $S(v)$ includes each element of $E^\rightarrow_v$ exactly once.
	\item The set of edges that are the bottom of the stacks other than at $z$ forms a spanning in-arborescence of $G$ rooted at $z$.
	% last-exist edges $\mathscr{L}(S)$ from vertices other than $z$ is a spanning in-arborescence of $G$ rooted at $z$. Equivalently, $\mathscr{L}(S)$ does not contain any cycles.
\end{enumerate}
In particular, the number of Eulerian paths from $o$ to $z$ is equal to
\[
|\mathcal{A}_z^{\operatorname{in}}(G)|\cdot \operatorname{out-deg}(z)! \prod_{v\in V \setminus\{z\}} (\operatorname{out-deg}(v)-1)!,
\]
where $\mathcal{A}_z^{\operatorname{in}}(G)$ denotes the set of spanning in-arborescences of $G$ rooted at $z$.
\end{theorem}

While this theorem is typically used to \emph{enumerate} Eulerian paths, we will instead use it to \emph{sample} uniform Eulerian paths, as part of our theoretical analysis of uniform random Eulerian paths: It says that we can sample a uniform random Eulerian path by first sampling a uniform random spanning in-arborescence rooted at $z$, then sampling the remaining edges of the stack using independent uniform random permutations. The uniform random arborescence can in turn be sampled using \emph{Wilson's algorithm} \cite{wilson1996generating} (as we discuss further in Sections \ref{subsec:Wilsons_background} and \ref{subsec:WUA_definition}), which again will play an important theoretical role in our analysis of random Eulerian tours and partially exchangeable processes. The BEST theorem was also used to study finite partially-exchangeable sequences \cite{zaman1984urn,zaman1986finite} in work that predated the introduction of Wilson's algorithm.

While we use the BEST theorem to reduce questions about Eulerian paths to questions about spanning trees, related ideas have also been used to study uniform spanning trees themselves in \cite{hu2021reverse}.

\subsection{Wilson's algorithm}
\label{subsec:Wilsons_background}

We now briefly review \emph{Wilson's algorithm} \cite{wilson1996generating}, referring the reader to \cite[Chapters 4.1 and 10.1]{MR3616205} for further background. We will discuss the algorithm only for digraphs, but everything we write also applies to general Markov chains. 

\medskip

\noindent
\textbf{Random walk.}
Given a locally finite digraph $G=(V,E)$, the \textbf{random walk} on $G$ is the Markov chain which, at each step, picks a uniform random element of the set of edges emanating from its current location, and crosses that edge. Note that the random walk naturally defines a random \emph{path} as well as a random sequence of visited vertices. We will write $\mathbf{P}_x=\mathbf{P}_x^G$ for the law of the random walk on $G$ started at the vertex $x$.

\medskip

\noindent \textbf{Loop erasure.}
Let $G=(V,E)$ be a digraph and let $\gamma$ be a transient path in $G$ indexed by either $\{0,1,2,\ldots\}$ or $\{0,1,\ldots,n\}$ for some $n\geq 0$. The \textbf{loop-erasure} $\textsf{LE}( \gamma )$ of $\gamma$ is formed by erasing cycles from $\gamma$ chronologically as they are created, and is defined formally by setting $\textsf{LE}(\gamma)_i  = \gamma_{t_i}$ and $\textsf{LE}(\gamma)_{i,i+1}  = \gamma_{t_i,t_i+1}$ where the times $t_i$ are defined recursively by $t_0 = 0$ and $t_i = 1+ \max \{ t \geq t_{i-1} : \gamma_t = \gamma_{t_{i-1}}\}$. The loop erasure of random walk is known as \textbf{loop-erased random walk} and has been studied extensively since the pioneering work of Lawler \cite{lawler1980self}.

\medskip

\noindent \textbf{Wilson's algorithm.}
Suppose we have a finite digraph $G=(V,E)$ with a distinguished sink vertex $z$. If every vertex $u$ in $G$ is connected to $z$ by a path that starts at $u$ and ends at $z$, then the set of spanning in-arborescences of $G$ rooted at $z$ is non-empty; we will assume this is the case for the remainder of the discussion. Wilson's algorithm allows us to sample a uniform random spanning in-arborescence of $G$ rooted at $z$ using loop-erased random walks in the following way. Let $V=\{v_0,v_1,\ldots,v_N\}$ be an enumeration of $V\setminus \{z\}$
and define a sequence $( A_i )_{i= 0}^n$ of in-arborescences rooted at $z$  as follows:
\begin{enumerate}
\item Let $A_0$ consists of the vertex $z$ and no edges.
\item Given $A_j$ for some $0\leq j \leq n$, start an independent random walk on $G$ from $v_{j+1}$ stopped when it hits the set of vertices already included in $A_j$. (If $v_{j+1}$ is already included in $A_j$ this path will have length zero.)
\item Form the loop-erasure of this random walk path and let $A_{j+1}$ be the union of $A_j$ with this loop-erased path.
% \item Let $\F = \bigcup \F_i$.
\end{enumerate}
The condition that every vertex of $G$ is connected to $z$ by a directed path ensures that the algorithm terminates in finite time almost surely.
We set $A=A_n = \bigcup_{j=0}^nA_j$ to be the final output of the algorithm, which is a spanning in-arborescence of $G$ rooted at $z$. Wilson \cite{wilson1996generating} proved the miraculous fact that the arborescence $A$ generated by this algorithm is a uniform random element of the set of spanning in-arborescences of $G$ rooted at $z$; in particular, its law does not depend on which enumeration of $V$ we use. (For \emph{undirected} graphs, we can then forget the orientation of edges in the arborescence to get a \emph{uniform spanning tree} of the graph, whose law does not depend on the choice of root. For directed graphs the choice of root is important.)

\medskip

Wilson's algorithm was extended to \emph{infinite} undirected graphs in the seminal work of Benjamini, Lyons, Peres, and Schramm \cite{benjamini2001special}, and plays a central role in the theoretical analysis of uniform spanning trees of infinite lattices. We will discuss how this theory can be extended to the directed case in Section~\ref{subsec:WUA_definition}.

\section{Proof}

In this section we prove Theorem~\ref{thm:DF_Gibbs}. We begin by proving that sourced Eulerian digraphs are always transient for simple random walk in Section~\ref{subsec:transience}. This is important since Wilson's algorithm rooted at infinity has potentially much more complicated behaviour in the recurrent case, where the infinite loop-erased random walk may not be uniquely defined \cite{van2023number}. In Section~\ref{subsec:WUA_definition}, we define the \emph{wired uniform spanning in-arborescence} of an infinite transient digraph as a limit over exhaustions and prove that it can be sampled using Wilson's algorithm rooted at infinity, establishing directed analogues of classical theorems from \cite{Pemantle,benjamini2001special}. Then, in Section~\ref{subsec:Gibbs_uniqueness}, we define Gibbs measure for spanning in-arborescences rooted at infinity and prove that any such measure supported on arboreta with one-ended components must be equal to the wired uniform spanning in-arborescence; this is the technical heart of the paper and draws close inspiration from the second author's work on indistinguishability \cite{hutchcroft2020indistinguishability,hutchcroft2017indistinguishability}. Finally, in Section~\ref{subsec:main_proof} we explain how this theorem implies Theorem~\ref{thm:DF_Gibbs} and state Theorem~\ref{thm:DF_explicit}, which gives an explicit identification of the unique proper Gibbs measure on Eulerian paths when it exists.

\subsection{Transience of sourced Eulerian digraphs}
\label{subsec:transience}

We will make vital use of the following simple but striking lemma concerning the random walk on a sourced Eulerian digraph. As usual, a digraph $G$ is said to be \textbf{transient} if the random walk on $G$ visits each vertex at most finitely often almost surely.
 % uniformly at random from 

\begin{lemma}
\label{lem:transience}
Every infinite sourced Eulerian digraph is transient. 
\end{lemma}

Before proving the lemma, let us first give some relevant definitions. The \textbf{communicating classes} of $G$ are the equivalence classes of vertices that can be reached from each other via directed paths in $G$; it is a basic result in the theory of Markov chains that the elements of each communicating class are either all transient or all recurrent.
Given a digraph $G$, we write $G^\leftarrow$ for the digraph with the same vertex and edge set as $G$, but where we swap the head of each edge with its tail. In particular, this operation exchanges the in-degree and out-degree of each vertex. 
% Thus, if $G$ is a finite Eulerian digraph with source $o$ and sink $z$ then $G^\leftarrow$ is a finite Eulerian digraph with source $z$ and sink $o$, and there is a bijection from Eulerian paths from $o$ to $z$ in $G$ with Eulerian paths from $z$ to $o$ in $G^\leftarrow$ defined by reversing the order of edges in the path.

\begin{proof}[Proof of Lemma~\ref{lem:transience}]
We will prove more generally that if $G$ is a connected digraph for which every vertex has out-degree at least in-degree and this inequality is strict for at least one vertex then $G$ is transient.
It suffices to prove the claim in the case that the digraph $G$ is irreducible. Indeed, if $V_0$ is a recurrent communicating class of $G$ then every edge emanating from a vertex of $V_0$ must have its head in $V_0$ also (since otherwise $V_0$ would not be recurrent), so that the subdigraph of $G$ induced by $V_0$ also has $\operatorname{out-deg}\geq \operatorname{in-deg}$ at every vertex, with strict inequality at at least one vertex: If the latter condition did not hold, $V_0$ would be a connected component of $V$ not equal to $V$, which is not possible since $G$ was assumed to be connected.

When $G$ is irreducible, it suffices to prove that if $o$ is a vertex of $G$ such that $\operatorname{out-deg}(o)>\operatorname{in-deg}(o)$ then the random walk on $G$ started at $o$ has positive probability never to return to $o$.
Consider the digraph $G^\leftarrow$ formed by reversing the direction of every edge in $G$ as above.
Write $\mathbf{P}_o$ and $\mathbf{P}_o^\leftarrow$ for the law of the simple random walk on $G$ and $G^\leftarrow$, respectively, started at $o$.
We can write the probability that the random walk on $G$ returns to $o$ as 
\[
\mathbf{P}_o(\text{return to $o$}) = \sum_{n=1}^\infty \sum_{\gamma} \prod_{i=0}^{n-1} \frac{1}{\operatorname{out-deg}(\gamma_i)},
\]
where the sum is over all paths $\gamma$ of length $n$ in $G$ that start and end at $o$ and do not visit $o$ at any intermediate step. Since the path $\gamma^\leftarrow$ defined by $\gamma_i^\leftarrow=\gamma_{n-i}$ and $\gamma^\leftarrow_{i,i+1}=\gamma_{n-i-1,n-i}$ is a path of the same form in $G^\leftarrow$, we have by the same reasoning that
\begin{multline*}
\mathbf{P}_o^\leftarrow (\text{return to $o$}) =\sum_{n=1}^\infty \sum_{\gamma} \prod_{i=1}^n \frac{1}{\operatorname{in-deg}(\gamma_i)} \\\geq \frac{\operatorname{out-deg}(o)}{\operatorname{in-deg}(o)}\sum_{n=1}^\infty \sum_{\gamma} \prod_{i=1}^n \frac{1}{\operatorname{out-deg}(\gamma_i)}  = \frac{\operatorname{out-deg}(o)}{\operatorname{in-deg}(o)} \mathbf{P}_o(\text{return to $o$}),
\end{multline*}
where the middle inequality follows since every vertex has in-degree at most out-degree. Since the probability $\mathbf{P}^\leftarrow_o (\text{return to $o$})$ is at most $1$, it follows that
\[
\mathbf{P}_o(\text{return to $o$}) \leq \frac{\operatorname{in-deg}(o)}{\operatorname{out-deg}(o)} < 1
\]
as claimed. 
\end{proof}

\begin{remark}
It is a theorem of Benjamini, Gurel-Gurevich, and Lyons \cite{MR2308594} and Benjamini and Hermon \cite{MR4059006} that if $Z$ is a transient Markov chain then the \emph{trace} of $Z$, that is, the \emph{undirected} graph in which the number of edges between any two states is equal to the number of times $Z$ transitions between the two states in either order, is recurrent almost surely. Lemma~\ref{lem:transience} shows that the opposite is true (for deterministic reasons!) when one considers the random walk on the \emph{directed} trace.
\end{remark}

\subsection{Wilson's algorithm rooted at infinity and the wired uniform spanning arboretum}
\label{subsec:WUA_definition}

Let $G=(V,E)$ be an infinite, connected, locally finite digraph.
Our goal in this section is to define the \emph{wired uniform in-arboretum} (WUA) of $G$, which is defined as a limit of uniform random arborescences on finite digraphs, and explain how the WUA can be sampled using \emph{Wilson's algorithm rooted at infinity}. 

\medskip

Let us begin with some relevant definitions.
 An \textbf{exhaustion} of $G$ is a sequence $(V_n)_{n\geq 1}$ of finite subsets of $V$ such that $V_1 \subseteq V_2 \subseteq \cdots$ and $\bigcup_{n\geq 1} V_n = V$. Given an exhaustion $(V_n)_{n\geq 1}$, we write $G_n^*$ for the digraph formed by identifying all the vertices of $V\setminus V_n$ into a single boundary vertex $\partial_n$ and deleting all edges that have both endpoints at $\partial_n$, so that the resulting digraph is locally finite. The edges of $G_n^*$ are naturally identified with the set of edges of $G$ that have at least one endpoint in $V_n$. When $G$ is transient, every vertex of $G_n^*$ is connected to $\partial_n$ by a directed path, so that $G_n^*$ has a non-empty set of in-arborescences rooted at $\partial_n$. We write $\mathsf{UA}_{G_n^*,\partial_n}$ for the law of the uniform random spanning in-arborescence on $G_n^*$ rooted at $\partial_n$, which we consider as a measure on subsets of $E$. The following proposition states that there is a well-defined limit of the uniform random arborescences on these graphs, and relates it to \emph{Wilson's algorithm rooted at infinity}, which we define below.

\begin{proposition}[The wired uniform spanning in-arboretum]
\label{prop:wired_well_defined}
Let $G=(V,E)$ be a infinite, connected, locally finite transient digraph and let $(V_n)_{n\geq 1}$ be an exhaustion of $G$ by finite sets. Then the limit
\[
\mathsf{WUA}_G = \lim_{n\to \infty} \mathsf{UA}_{G_n^*,\partial_n}
\]
is well-defined and does not depend on the choice of exhaustion $(V_n)_{n\geq 1}$. Moreover, the random in-arborescence rooted at infinity that is generated using \emph{Wilson's algorithm rooted at infinity} has law $\mathsf{WUA}_G$.
\end{proposition}

The analogous theorems in the \emph{undirected} case are due to Pemantle \cite{Pemantle}, who proved implicitly that the wired uniform spanning forest is well-defined on any infinite, connected, locally finite graph (via a method that does not generalize to the directed case) and Benjamini, Lyons, Peres, and Schramm \cite{benjamini2001special}, who showed that in the transient case this measure can be sampled from using Wilson's algorithm rooted at infinity.

\begin{remark}
As in the undirected case, it is possible for the limiting measure $\mathsf{WUA}_G$ to be supported on configurations with multiple infinite connected components (that is, the measure can be supported on arboreta rather than arborescences). In the undirected case, Pemantle \cite{Pemantle} proved that the wired uniform spanning forest of $\Z^d$ is connected a.s.\ if and only if $d\leq 4$, and Benjamini, Lyons, Peres, and Schramm \cite{benjamini2001special} proved more generally that the wired uniform spanning forest of a graph is connected a.s.\ if and only if two independent random walks on the graph intersect infinitely often a.s. Considering $\Z^d$ as a digraph in which each neighboring pair of vertices are connected by a pair of directed edges, one in each direction, shows that the same behaviour is possible on digraphs.
\end{remark}

\noindent \textbf{Wilson's algorithm rooted at infinity.} We now define Wilson's algorithm rooted at infinity on transient digraphs; the construction is exactly the same as the construction for \emph{undirected} graphs as given in \cite{benjamini2001special}.
Let $G=(V,E)$ be an infinite, connected, locally finite, transient digraph, and let $V=\{v_0,v_1,\ldots\}$ be an enumeration of $V$. We define a sequence of in-arboreta
% and define a sequence 
$( A_i )_{i\geq 0}$ as follows:
\begin{enumerate}
\item Let $A_0$ be the empty digraph, with no vertices or edges.
\item Given $A_j$ for some $j \geq 0$, start an independent random walk on $G$ from $v_{j+1}$ stopped if and when it hits the set of vertices already included in $A_j$. (It is possible for the walk to run forever, as will certainly be the case in the first step. If $v_{j+1}$ is already included in $A_j$ the path will have length zero.) 
\item Form the loop-erasure of this stopped random walk path and let $A_{j+1}$ be the union of $A_j$ with this loop-erased path.
% \item Let $\F = \bigcup \F_i$.
\end{enumerate}
We set $A= \bigcup_{j=0}^\infty A_j$ to be the final output of the algorithm, which is a spanning in-arboretum of $G$ rooted at infinity. It will follow from Proposition~\ref{prop:wired_well_defined} that this random arboretum has law $\mathsf{WUA}_G$, and in particular that its law does not depend on the choice of enumeration of $V$.
Note that when we run Wilson's algorithm rooted at infinity starting with the vertices $v_1,\ldots,v_k$, the arboretum we generate after the first $k$ steps is precisely the union of the futures of $v_1,\ldots,v_k$ in the full arboretum.

\begin{proof}[Proof of Proposition~\ref{prop:wired_well_defined}]
This follows by an identical proof to the usual proof that Wilson's algorithm rooted at infinity generates the wired uniform spanning forest of an undirected\footnote{For undirected graphs, one typically proves that the limit measure is well-defined using relations to effective resistance and Rayleigh's monotonicity principle, which are not available in the directed setting. This was the approach taken in the original work of Pemantle \cite{Pemantle}, and predates Wilson's algorithm. This proof also establishes the existence of the \emph{free} uniform spanning forest measure.} transient graph: see e.g.\ \cite[Proposition 10.1]{MR3616205}. 
The same argument has previously been used to show that the \emph{oriented} wired uniform spanning forest is well-defined and can be generated using Wilson's algorithm on any undirected transient graph \cite{benjamini2001special,hutchcroft2016wired}.

\medskip

We now provide a brief proof to keep the paper self-contained.
First note that if $\gamma$ is a transient path and $\gamma^k$ denotes the path formed from the first $k$ steps of $\gamma$, then $\mathsf{LE}(\gamma^k)$ converges to $\mathsf{LE}(\gamma)$ as $k \to \infty$ in the sense that for any $m\geq 1$, the two loop-erased paths will have the same initial segment of length $m$ for all sufficiently large $k$. This is an immediate consequence of the definition of the loop-erasure.

\medskip

Fix a finite list of edges $e_1,\ldots,e_M \in E$ and let $u_1,\ldots,u_M$ be the vertices these edges emanate from. If $n$ is sufficiently large that $u_1,\ldots,u_M \in V_n$, we can sample the uniform in-arborescence of $G_n^*$ rooted at $\partial_n$ using Wilson's algorithm, starting with the vertices $u_1,\ldots,u_M$. Moreover, the intersection of this random arborescence with the set $\{e_1,\ldots,e_M\}$ is determined by the first $M$ steps of the algorithm. At each step of the algorithm, instead of using the random walk on $G_n^*$, we can instead stop the random walk on $G$ when it first leaves $V_n$; this does not change the law of the resulting arborescence, since the two walks can be coupled to be identical (as sequences of edges) after appropriate identification of the edges of $G_n^*$ with edges of $G$. In particular, we can run Wilson's algorithm using the \emph{same} infinite random walk paths $X^1,\ldots,X^M$ on $G$ started at the vertices $u_1,\ldots,u_M$ for every $n\geq 1$; the effect of working in finite volume is simply to stop these walks when they leave $V_n$. Write $X^{i,n}$ for the walk started at $u_i$ and stopped when it leaves $V_n$, let 
$\tau_i^n$ be the first time that $X^{i,n}$ reaches the set of vertices included in the arborescence generated by the random walks $X^{j,n}$ with $j<i$, and let $L^{i,n}$ be the loop-erasure of $X^{i,n}$ run up to time $\tau_i^n$. Then
\begin{align*}
\mathsf{UA}_{G_n^*,\partial_n}(e_1,\ldots,e_M \text{ are all included}) = \mathbb{P}\left( e_1,\ldots,e_M \text{ are each crossed by one of the paths $L^{1,n},\ldots,L^{M,n}$} \right).
\end{align*}
Let $L^1,\ldots,L^m$ be the analogous loop-erased random walks generated when running Wilson's algorithm rooted at infinity with the walks $X^1,\ldots,X^M$. It follows by induction on $i$ that $L^{i,n}$ converges to $L^i$ as $n\to\infty$, so that the limit
\begin{align*}
\lim_{n\to\infty}\mathsf{UA}_{G_n^*,\partial_n}(e_1,\ldots,e_M \text{ are all included}) = \mathbb{P}\left( e_1,\ldots,e_M \text{ are each crossed by one of the paths $L^{1},\ldots,L^{M}$} \right)
\end{align*}
is well-defined independently of the choice of exhaustion. Since $e_1,\ldots,e_M$ were arbitrary and the probabilities of all other cylinder events can be computed using inclusion-exclusion, this implies both that the measure $\mathsf{WUA}_G$ is well-defined and that it can be generated using Wilson's algorithm rooted at infinity.
\end{proof}

\begin{remark}
While the \emph{wired uniform spanning forest} of an \emph{undirected} graph is always well-defined independently of the choice of exhaustion, the analogous statement for \emph{directed} graphs is not always true without the assumption of transience. For example, consider the digraph with vertex set $\Z$ and with a pair of oriented edges in each direction between each two adjacent integers. If we take the limit over the exhaustion $([-n,n^2])_{n\geq 1}$ then the limit will be supported on the arborescence consisting of every left-directed edge, while if we use exhaustion $([-n^2,n])_{n\geq 1}$ then the limit will be supported on the arborescence consisting of every right-directed edge. For general recurrent graphs, having a well-defined wired uniform arborescence measure is closely related to the uniqueness of the \emph{harmonic measure from infinity} and \emph{potential kernel}; see \cite{BvE21,van2023number} for further details.
\end{remark}

\subsection{Uniqueness of Gibbs measures with one-ended components}
\label{subsec:Gibbs_uniqueness}

The goal of this subsection is to prove the following theorem. The definition of Gibbs measures on the set $\mathcal{A}_\infty^\mathrm{in}(G)$ of spanning in-arborescences of $G$ rooted at infinity will be given after the statement.

\begin{theorem}
\label{thm:one_end_uniqueness}
Let $G=(V,E)$ be an infinite, locally finite, connected digraph. 
If $G$ is transient then there is at most one Gibbs measure on $\mathcal{A}_\infty^\mathrm{in}(G)$ that is supported on configurations in which every component is one-ended. Moreover, if such a measure exists then it must be equal to the measure $\mathsf{WUA}_G$.
\end{theorem}

\begin{remark}
The analogue of Theorem~\ref{thm:one_end_uniqueness} for \emph{undirected} graphs follows trivially from the fact that the WUSF is stochastically minimal among all Gibbs measures for the UST. Indeed, any UST Gibbs measure that is not equal to the WUSF can be coupled to with the WUSF in such a way that it strictly contains the WUSF with positive probability, and on this event there must exist a tree with more than one end. 
\end{remark}

\begin{remark}
Theorem~\ref{thm:one_end_uniqueness} does \emph{not} imply that the measure $\mathsf{WUA}_G$ is supported on configurations in which every component is one-ended, and indeed there do exist transient undirected graphs in which the WUSF is supported on components with multiple ends. 
There is a rich history of works studying the number of ends of uniform spanning forest components of undirected graphs \cite{AldousLyonsUnimod2007,benjamini2001special,BvE21,van2023number,hutchcroft2016wired,hutchcroft2019uniform,MR3773383,LyonsMorrisSchramm2008}. In our primary application to the Diaconis--Freedman conjecture, one-endedness holds trivially by Lemma~\ref{lem:entrance_exit_basic_properties}.
\end{remark}

\begin{remark}
Theorem~\ref{thm:one_end_uniqueness} can be used to deduce that the measure $\mathsf{WUA}_G$ is tail-trivial whenever its components are one-ended a.s., since otherwise we could condition on a tail event to get a distinct Gibbs measure with this property. For the wired uniform spanning forest on an \emph{undirected} graph, tail triviality of the measure was proven by Benjamini, Lyons, Peres, and Schramm \cite{benjamini2001special} using monotonicity properties that derive from relations with electrical networks and do not generalize to the directed case. The proof of Theorem~\ref{thm:one_end_uniqueness} is inspired by arguments used to prove \emph{indistinguishability theorems} for the wired uniform spanning forest \cite{hutchcroft2020indistinguishability,hutchcroft2017indistinguishability}, which can be thought of as a strong form of tail triviality. In contrast to the usual proof of tail triviality for the WUSF, these arguments are based directly on the analysis of Wilson's algorithm and are better suited to the nonreversible setting. It should be possible to use these arguments to prove analogous indistinguishability theorems for the wired uniform spanning in-arboretum, but we do not pursue this here.
\end{remark}

We now define Gibbs measures on $\mathcal{A}_\infty^\mathrm{in}(G)$, beginning with the definition of a \emph{finite approximation} to $G$.

\medskip

\noindent \textbf{Finite approximations.} 
Given two digraphs $G=(V,E)$ and $G'=(V',E')$, a \textbf{partial isomorphism} $\Phi$ from $G$ to $G'$ consists of the data $(A,A',F,F',\Phi_V,\Phi_E)$, where $A\subseteq V$ and $A'\subseteq V'$ are sets of vertices, $F\subseteq E$ and $F'\subseteq E'$ are sets of edges, and $\Phi_V:A\to A'$ and $\Phi_E:F\to F'$ are bijections such that $\Phi_E(e)^+=\Phi_V(e^+)$ and $\Phi_E(e)^-=\Phi_V(e^-)$ for every edge $e\in F$. In particular, the sets $A$ and $A'$ must include both endpoints of every edge in $F$ and $F'$ respectively.
% is a pair of bijections between some subset of the vertices of $G$ and some subset of
Let $G$ be an infinite, locally finite, connected digraph. A \textbf{finite approximation} to $G$ is a sequence $(G_n,z_n,\Phi_n)_{n \geq 1}$ where, for each $n\geq 1$, $G_n$ is a finite digraph, $\Phi_n$  is a partial isomorphism from $G_n$ to $G$, and $z_n$ is a vertex of $G_n$ satisfying the following properties:
\begin{enumerate}
	\item (Divergent radius of identification.) For each finite set of vertices $W$ of $G$ there exists $N<\infty$ such that for every $n\geq N$, the set $W$ is contained in the range of $\Phi_n$ and $\Phi_n$ maps the set of edges incident to $\Phi_n^{-1}(W)$ bijectively onto the set of edges incident to $W$.
	\item (The sink goes to infinity.) For each finite set of vertices $W$ of $G$, there exists $N<\infty$ such that $z_n \notin \Phi_n^{-1}(W)$ for every $n\geq N$.
\end{enumerate}
For example, if $(V_n)_{n\geq 1}$ is an exhaustion of $G$, the graphs $G_n^*$ together with the boundary vertices $\partial_n$ naturally form a finite approximation to $G$, where the partial isomorphism $\Phi_n$ identifies the vertices of $V_n$ and edges with both endpoints in $V_n$ between the two graphs. 

\medskip

\noindent \textbf{Gibbs measures.}
We define a probability measure $\mu$ on the set $\mathcal{A}_\infty^\mathrm{in}(G)$ of spanning in-arborescences rooted at infinity  to be a \textbf{Gibbs measure} if it can be written 
\[
\mu = \lim_{n\to \infty} \int \Phi_n^* (\mathsf{UA}_{G_n,z_n}) \mathrm{d} \nu
\]
where $\nu$ is a probability measure on finite approximations $(G_n,z_n,\Phi_n)$ to $G$ and $\Phi_n^*$ denotes the pushforward map associated to the partial isomorphism $\Phi_n$. That is, Gibbs measures on $\mathcal{A}_\infty^\mathrm{in}(G)$ are weak limits of uniform random spanning arborescences on finite approximations to $G$, where we allow these finite approximations to themselves be randomized and use the partial isomorphisms to consider the measures as measures on sets of edges in $G$. For example, the measure $\mathsf{WUA}_G$ constructed in Proposition~\ref{prop:wired_well_defined} is a Gibbs measure on $\mathcal{A}_\infty^\mathrm{in}(G)$.

\begin{remark}
Following \cite{halberstam2023uniqueness}, it is also possible to give an equivalent axiomatic definition of Gibbs measures on $\mathcal{A}_\infty^\mathrm{in}(G)$ in terms of the Dobrushin–Lanford–Ruelle (DLR) equations for suitable \emph{augmentations} of the measure. We do not pursue this here.
\end{remark}

\begin{proof}[Proof of Theorem~\ref{thm:one_end_uniqueness}]
Let $((G_n,z_n,\Phi_n))_{n\geq 1}$ be a (possibly random) finite approximation to $G$ such that the law $\int \Phi_n^* (\mathsf{UA}_{G_n,z_n}) \mathrm{d} \nu$ of the image under $\Phi_n$ of the uniform spanning in-arborescence rooted at $z_n$ on $G_n$ converges weakly to some measure $\mu$ on $\mathcal{A}_\infty^\mathrm{in}(G)$ that is supported on configurations in which every component is one-ended.  We wish to show that this measure must be equal to $\mathsf{WUA}_G$; we will do this by showing that it can be sampled using Wilson's algorithm rooted at infinity.
As in the proof of Proposition~\ref{prop:wired_well_defined}, we will use the fact that if $\gamma$ is a transient path and $\gamma^k$ denotes the path formed from the first $k$ steps of $\gamma$, then $\mathsf{LE}(\gamma^k)$ converges to $\mathsf{LE}(\gamma)$ as $k \to \infty$ in the sense that for any $m\geq 1$, the two loop-erased paths will have the same initial segment of length $m$ for all sufficiently large $k$.

\medskip

Let $F$ be a finite set of edges in $G$, let $K$ denote the set of tails of edges in $F$, and let $(V_m)_{m\geq 1}$ be an exhaustion of $V$ by finite sets. For each $m\geq 1$, let $E_m$ be the set of edges with at least one endpoint in $V_m$, let $\bar V_m \supseteq V_m$ denote the set of endpoints of edges of $E_m$, and let $\partial V_m = \bar V_m \setminus V_m$, which is finite for each $m\geq 1$ since $G$ is locally finite. For each $n,m\geq 1$ let $\mathscr{D}_{n,m}$ denote the event that $E_m$ and $\bar V_m$ are both contained in the range of the partial isomorphism $\Phi_n$ and that for each vertex $v$ in $\Phi_n^{-1}(\bar V_m)$, the partial isomorphism $\Phi_n$ maps the edges incident to $v$ bijectively to the edges incident to the image of $v$. Thus, when $\mathscr{D}_{n,m}$ holds, the random walk on $G$ and $G_n$ can be coupled so that (after identification via $\Phi_n$) they coincide up until they first leave the set $V_m$.

\medskip

Let $A$ be a random variable in $\mathcal{A}_\infty^\mathrm{in}(G)$ with law $\mu$. For each $n\geq 1$ let $A_n$ be a sample of the uniform spanning in-arborescence of $G_n$ rooted at $z_n$, and let $Z_{n,m}$ be the intersection with $V_m$ of the image under $\Phi_n$ of the set of vertices of $G_n$ that belong to the future of some vertex not in $\Phi_n^{-1}(V_m)$. When $\mathscr{D}_{n,m}$ holds, $Z_{n,m}$ is equal to the intersection with $V_m$ with the image under $\Phi_n$ of the set of vertices of $G_n$ that belong to the future of some vertex in $\Phi_n^{-1}(\partial V_m)$: This is because the future of any vertex not in $\Phi_n^{-1}(V_m)$ must pass through $\Phi_n^{-1}(\partial V_m)$ in order to reach $\Phi_n^{-1}(V_m)$. Since $\partial V_m$ is finite, it follows that $Z_{n,m}$ converges in distribution as $n\to\infty$ to $Z_m$, the set of vertices in $V_m$ that belong to the future in $A$ of some vertex in $\mathcal{A}_\infty^\mathrm{in}(G)$. The statement that $A$ has one-ended components almost surely is equivalent (by continuity of measure) to the statement that $Z_m$ converges weakly to the empty set as $m\to \infty$ in the sense that if $W$ is any finite subset of $V$ then the probability that $Z_m$ intersects $W$ converges to zero as $m\to \infty$.

\medskip

Fix an enumeration $\{v_1,\ldots,v_\ell\}$ of the set $K$ and let $n,m\geq 1$.
Once we have sampled $(G_n,z_n,\Phi_n)$, we can sample the uniform spanning in-arborescence $A_n$ using Wilson's algorithm rooted at $z_n$. We do this using an enumeration of the vertex set of $G_n$ that is broken into three stages: We first enumerate the vertices not belonging to $\Phi_n^{-1}(V_m)$, then use the fixed enumeration $\Phi_n^{-1}(v_1),\ldots,\Phi_n^{-1}(v_\ell)$ of the set $\Phi_n^{-1}(K)$, then enumerate the remaining vertices in an arbitrary way.
In the first stage, in which we are using the vertices that do not belong to $\Phi_n^{-1}(V_m)$, we generate the part of the arborescence that is equal to the union of the futures of all the vertices in this set.

\medskip

Now consider the second stage of the algorithm. All the walks used in this stage will take place inside the set $\bar V_m$, so that when the event $\mathscr{D}_{n,m}$ holds they all occur inside the part of $G_n$ that is identified with $G$ by the partial isomorphism $\Phi_n$. As such, we can safely consider all these walks to take place on $G$ rather than $G_n$ whenever the event $\mathscr{D}_{n,m}$ holds, and we will drop the partial isomorphism $\Phi_n$ from our notation.
Let $X^1,\ldots,X^{\ell}$ be independent random walks on $G$ started at $v_1,\ldots,v_\ell$. For each $1\leq i \leq \ell$ and $n,m\geq 1$ such that $\mathscr{D}_{n,m}$ holds, let 
$\tau_i^{n,m}$ be the first time that $X^{i}$ reaches the set of vertices included in the arborescence generated by an earlier part of Wilson's algorithm: This includes both the set of vertices $Z_{n,m}$ generated in the first stage of the algorithm and the additional loop-erased paths associated to the walks $X^{j,n}$ with $j<i$. Let $L^{i,n,m}$ be the loop-erasure of $X^{i}$ run up to time $\tau_i^{n,m}$. Then, letting $\nu_n$ denote the joint law of $(G_n,z_n,\Phi_n)$ and the uniform random spanning in-arborescence $A_n$ of $G_n$ rooted at $z_n$,
\begin{equation*}
\nu_n(F \subseteq A_n \mid \mathscr{D}_{n,m}) 
= \mathbb{P}\left(\text{Every edge in $F$ crossed by one of the paths $L^{1,n,m},\ldots,L^{\ell,n,m}$} \mid \mathscr{D}_{n,m} \right).
\end{equation*}
As in the proof of Proposition~\ref{prop:wired_well_defined}, we can induct over $\ell$ to take the limit as $n\to\infty$ and obtain that
\begin{equation*}
\mu(F \subseteq A) 
= \mathbb{P}\left(\text{Every edge in $F$ crossed by one of the paths $L^{1,m},\ldots,L^{\ell,m}$} \right),
\end{equation*}
where the loop-erased paths $L^{1,m},\ldots,L^{\ell,m}$ are formed by running Wilson's algorithm with the walks $X^1,\ldots,X^\ell$ but where we consider the set $Z_m$ to be included in the arborescence at time zero.
Since $Z_m$ converges weakly to the empty set as $m\to\infty$, we can  obtain, by similar reasoning, that 
the probability on the right hand side converges as $m\to \infty$ to the analogous probability for the standard implementation of Wilson's algorithm rooted at infinity. Since the finite set of edges $F$ was arbitrary, this implies that $\mu$ is equal to $\mathsf{WUA}_G$ as claimed. (As before, the probabilities of arbitrary cylinder events are determined by probabilities of the form $\mu(F \subseteq A)$ using inclusion-exclusion.) \qedhere
\end{proof}

\subsection{Proof of the main theorem}
\label{subsec:main_proof}

It remains only to deduce Theorem~\ref{thm:DF_Gibbs} from Theorem~\ref{thm:one_end_uniqueness} and the BEST theorem. We will in fact prove the following stronger form of the theorem, which gives an explicit identification of the unique proper Gibbs measure along with a method to sample from it. Note that the theorem implicitly makes use of Lemma~\ref{lem:transience}, which ensures that $G$ is transient and hence that Wilson's algorithm rooted at infinity is well-defined.

\begin{theorem}
\label{thm:DF_explicit}
Let $(G,o)$ be a sourced Eulerian digraph, and suppose that $\mu$ is a proper Gibbs measures on Eulerian paths from $o$ to $\infty$ in $G$. We can sample from $\mu$ using the following procedure:
\begin{enumerate}
\item First, sample the wired uniform spanning in-arboretum $A$ of $G$ rooted at infinity using Wilson's algorithm rooted at infinity as in Section~\ref{subsec:WUA_definition}. This arboretum will form the last-exit arboretum of the Eulerian path.
\item For each vertex $v$ of $G$, pick a uniform random ordering of the edges of $E_v^\rightarrow$ that do not belong to the arboretum $A$. These random orderings are chosen conditionally independently of one another given $A$.
\item Form a stack configuration on $G$ by placing the edges of the arboretum $A$ at the bottom of the stack at each vertex, and using the uniform random orderings constructed in Step 2 as the remaining part of the stack.
\item Use this stack configuration to construct a path starting at $o$ as described in Section~\ref{subsec:Euler_background}.
\end{enumerate}
\end{theorem}

Note that this theorem is an exact analogue of the method of sampling uniform random Eulerian paths on \emph{finite} digraphs using Wilson's algorithm and the BEST theorem. Note also that although this sampling procedure makes sense on \emph{any} sourced Eulerian digraph (or indeed on any transient locally finite digraph), it may fail to produce an Eulerian path in general; if the path produced by this method is not a.s.\ Eulerian, then $(G,o)$ does not admit any proper Gibbs measures on Eulerian paths.

\begin{proof}[Proof of Theorem~\ref{thm:DF_Gibbs}]
Let $(G,o)$ be an infinite sourced Eulerian digraph, let $\mu$ be a proper Gibbs measure on Eulerian paths in $G$, and let $X$ be a random variable with law $\mu$. For each $n\geq 0$, let $\mathcal{G}_n$ be the sigma-algebra generated by the steps taken by $X$ after time $n$, let $F_n$ be the set of edges that are crossed by $X$ in its first $n$ steps, and let $G_n$ be the subgraph of $G$ spanned by $F_n$, which it Eulerian with source $o$ and sink $X_n$. The Gibbs property implies that, conditional on $\mathcal{G}_n$, the first $n$ steps of $X$ are a uniform random Eulerian path from $o$ to $X_n$ in $G_n$. Moreover, since $\mu$ is proper, $(G_n,X_n,\Phi_n)_{n\geq 1}$ forms a finite approximation to $G$ where $\Phi_n$ is the inclusion map from $G_n$ to $G$. Letting $X^n$ denote the first $n$ steps of $X$, it follows from the BEST theorem that, conditional on $\mathcal{G}_n$, the last-exit arborescence $\mathscr{L}(X^n)$ is distributed as a uniform random spanning in-arborescence of $G_n$ rooted at $X_n$. Since $\mathscr{L}(X^n)$ converges to $\mathscr{L}(X)$ as $n\to \infty$, the law of $\mathscr{L}(X)$ is a Gibbs measure on $\mathcal{A}_\infty^\mathrm{in}(G)$. Since every component of $\mathscr{L}(X)$ is one-ended almost surely by Lemma~\ref{lem:entrance_exit_basic_properties}, it follows from Theorem~\ref{thm:one_end_uniqueness} that $\mathscr{L}(X)$ is distributed as the wired uniform spanning in-arboretum of $G$ rooted at infinity. 

\medskip

Now, it also follows from the BEST theorem that, conditional on $\mathcal{G}_n$, we can sample $X^n$ by first sampling the last-exit arborescence $\mathscr{L}(X^n)$, then sampling the remaining edges of each stack using independent uniform random permutations, then computing $X^n$ in terms of these stacks as in Section~\ref{subsec:Euler_background}. Since $\mathscr{L}(X^n)$ converges to $\mathscr{L}(X)$ and $G_n$ eventually contains every edge of $G$, the law of the stack configurations generating the path $X^n$ converge as $n\to\infty$ to the law specified in the statement of the theorem, where the last edges in the stack are given by the wired uniform spanning in-arborescence of $G$ rooted at infinity and the remaining edges in each stack are given by independent random permutations.  This implies that the law of $X^n$ also converges as $n\to\infty$ to that of the path generated by these stack configurations, concluding the proof. (To give a little more detail on this last point: we can couple the stack configurations generating $X^n$ with those described in the statement of the theorem so that if $K$ is any finite set of vertices then the two configurations agree on $K$ with probability tending to $1$. If we take $K$ to be the set of vertices reachable in $k$ steps from the origin, this forces the first $k$ steps of $X^n$ to coincide with those of the path generated by following the stacks in the coupled process. Taking $n\to\infty$ then $k\to\infty$ yields the claim.)
\end{proof}

\subsection*{Acknowledgements}
TH thanks Asaf Nachmias for introducing him to the problem in 2013. 
Both authors thank Persi Diaconis, Vadim Kaimanovich, and Russ Lyons for comments on a draft of the paper.
TH is supported by NSF grant DMS-2246494. 

\small{

	\bibliographystyle{abbrv}
	\bibliography{bibliography.bib}

}

	\end{document}